\documentclass{amsart}

\usepackage{amssymb,amsfonts}
\usepackage[all,arc]{xy}
\usepackage{enumitem}
\usepackage{mathrsfs}
\usepackage{verbatim}
\usepackage{microtype}
\usepackage{mathtools}
\usepackage[numbers]{natbib}

\newcounter{derpp}

\newtheorem{thm}{Theorem}[section]
\newtheorem*{thm*}{Theorem}
\newtheorem{cor}[thm]{Corollary}
\newtheorem{prop}[thm]{Proposition}
\newtheorem{lem}[thm]{Lemma}

\newtheorem*{conj*}{Conjecture I}

\newtheorem*{cg}{Cohen-Gabber Theorem}

\theoremstyle{definition}
\newtheorem{defn}[thm]{Definition}

\newtheorem{exmp}[thm]{Example}

\newtheorem{notns}[thm]{Notations}

\theoremstyle{remark}
\newtheorem{rem}[thm]{Remark}

\DeclareMathOperator{\Spec}{Spec}

\DeclareMathOperator{\Supp}{Supp}
\DeclareMathOperator{\ord}{ord}

\DeclareMathOperator{\Ass}{Ass}

\newcommand{\fm}{\mathfrak{m}}
\newcommand{\fp}{\mathfrak{p}}
\newcommand{\fq}{\mathfrak{q}}
\newcommand{\fa}{\mathfrak{a}}

\newcommand{\fN}{\mathfrak{N}}
\newcommand{\ds}{\displaystyle}

\newcommand{\mcm}{\mathcal{M}}

\newcommand{\GG}{\mathbf{G}}
\newcommand{\bx}{\mathbf{x}}
\newcommand{\by}{\mathbf{y}}
\newcommand{\bof}{\mathbf{f}}

\makeatletter
\let\c@equation\c@thm
\makeatother
\numberwithin{equation}{section}

\address{Department of Mathematics \\
University of Illinois at Chicago \\
851 S. Morgan St. \\
Chicago, IL 60607}
\email[Chris Skalit]{cskali2@uic.edu}
\title{Koszul Factorization and the Cohen-Gabber Theorem}
\author{C. Skalit}

\begin{document}
\begin{abstract}We present a sharpened version of the Cohen-Gabber theorem for equicharacteristic, complete local domains $(A,\mathfrak{m},k)$ with algebraically closed residue field and dimension $d>0$. Namely, we show that for any prime number $p$, $\operatorname{Spec} A$ admits a dominant, finite map to $\operatorname{Spec} k[[X_1, \cdots, X_d]]$ with generic degree relatively prime to $p$. Our result follows from Gabber's original theorem, elementary Hilbert-Samuel multiplicity theory, and a ``factorization'' of the map induced on the Grothendieck group $\mathbf{G}_0(A)$ by the Koszul complex.\end{abstract}
\maketitle

\section*{Introduction}

When $(A,\fm, k)$ is a complete, equicharacteristic, $d$-dimensional local domain, the familiar Cohen structure theorem says that there exists a power series subring $R = k[[X_1, \cdots, X_d]]$ over which $A$ is finite. When $k$ has characteristic zero, the map on fraction fields is automatically separable, but arranging for this condition in characteristic $p > 0$ requires a theorem of Gabber \cite[IV.2.1.1]{GTG} (see also \cite{KS} for an elementary proof):

\begin{cg}Let $(A,\fm,k)$ be a complete, local ring of characteristic $p > 0$. Suppose that $A$ is reduced and of equidimension $d$. Then there exists a subring $R \cong k[[X_1, \cdots, X_d]]$ such that the map $R \to A$ is finite and generically \'{e}tale. \end{cg}

The theorem is a crucial ingredient in the proof of the existence of so-called $\ell'$-alterations: given a variety $X/k$ and an $\ell \neq \operatorname{char} k$, there exists a regular, connected scheme $Y$, equipped with a proper map $Y \to X$ whose generic degree is finite and prime to $\ell$. We refer the reader to \cite{GTG} for a more precise statement and some generalizations to mixed-characteristic. From the perspective of commutative algebra, the existence of such a generically \'{e}tale Noether normalization simplifies arguments involving tight closures and test elements (see, for example, \cite[(6.3)]{HH90} and \cite[\S 4]{HH00}).

When the residue field is algebraically closed, we show that generic separability may be achieved in an even stronger sense:
\begin{thm*}Let $(A,\fm,k)$ be a complete, equicharacteristic local domain of positive dimension whose residue field is algebraically closed. Then for any prime $p > 0$, there exists a regular subring $R = k[[X_1, \cdots X_d]]$ such that $R \to A$ is finite and $p$ is relatively prime to the generic degree $[K(A):K(R)]$.
\end{thm*}
Note that when $A$ has characteristic $p > 0$, the condition that $p$ does not divide $[K(A):K(R)]$ will force the map to be generically \'{e}tale. Since the generic degree of the map is just the Hilbert-Samuel multiplicity of $A$ with respect to the parameter ideal $(X_1, \cdots X_d)$, the technical heart of the proof is the following statement, which appears below as Theorem \ref{gabber_mult}:

\begin{thm*}Let $(A,\fm)$ be a complete, reduced local ring of equidimension $d > 0$ with algebraically closed residue field. Suppose that the residue fields $k(\fp)$ for each minimal prime $\fp \subseteq A$ all have the same characteristic. Then for each prime $p \in \mathbb{Z}$, there exists a parameter ideal $I = (f_1, f_2, \cdots, f_d) \subseteq A$ such that $e_I(A)$ is relatively prime to $p$.
\end{thm*}

When $d = 1$, we prove the result via direct calculation (Lemma \ref{dim_one}). To reduce to the one-dimensional case, we carefully select $d-1$ parameters to cut down the dimension of $A$ and exploit the relationship between the Koszul complex and multiplicity with respect to parameter ideals. In particular, taking Koszul homology gives well-defined endomorphisms on the Grothendieck group of $A$ which can be decomposed into simple factors. We develop this notion of ``Koszul factorizations'' systematically in Section 1. The first two parts of Section 2 are dedicated to proving the main theorems. We close, in Section 2.3, by indicating some analogous results in the mixed-characteristic setting.

\section{Koszul Factorization}
\begin{notns}Fix a Noetherian ring $A$. For a closed subset $Y \neq \emptyset$ of $\Spec A$, we define $\mcm(Y)$ to be the Serre subcategory of finitely-generated $A$-modules $M$ whose support lies inside of $Y$. By $\GG_0(Y)$ we shall mean the Grothendieck group $K_0(\mcm(Y))$. Given a module $M \in \mcm(Y)$, we shall write $[M]$ for the corresponding class in $\GG_0(Y)$.
\end{notns}

\begin{defn}Let $\bx = (x_1, \cdots, x_n)$ be a collection of elements of $A$ whose vanishing set $V(\bx)$ has nonempty intersection with $Y$. Define $\Phi_{\bx}:\GG_0(Y) \to \GG_0(Y \cap V(\bx))$ via the rule
\[ [M] \mapsto \sum_{i=0}^{n}{(-1)^i [H_i(\bx,M)]} \]
where by $H_i(\bx,M)$ we mean the homology of the Koszul complex $K(\bx,M)$. The well-definedness of this map follows from the functoriality of the Koszul complex and the fact that $\bx \subseteq \operatorname{Ann}(H_i(\bx,M))$ \cite[IV, pp. 6-7]{Serre}.
\end{defn}

Since the Koszul complex $K(\bx, M)$ may be realized as the iterated tensor product $K(x_1, A) \otimes_A \cdots \otimes_A K(x_n, A) \otimes_A M$, we can use the associated spectral sequence to give a ``factorization'' of the map $\Phi_{\bx}$:

\begin{lem}\label{composefilter}Let $\bx = (x_1, \cdots, x_m)$, $\by = (y_1, \cdots, y_n)$ and put \newline $\bx + \by = (x_1, \cdots, x_m, y_1, \cdots, y_n)$. Fix some closed $Y \subseteq \Spec A$.
There is then a commutative diagram
\[\xymatrix{
\GG_0(Y) \ar[r]^{\Phi_{\by}} \ar[dr]_{\Phi_{\bx + \by}} & \GG_0(Y \cap V(\by)) \ar[d]^{\Phi_{\bx}} \\
																												& \GG_0(Y \cap V(\bx + \by))
}\]
\end{lem}
\begin{proof}For any $M \in \mcm(Y)$, the Koszul complex $K(\bx + \by,M)$ may be realized as the total complex of $K(\bx,A) \otimes_A K(\by,M)$. By considering the spectral sequence of the double complex
\[E^2_{pq} = H_p(\bx,H_q(\by,M)) \Rightarrow H_{p+q}(\bx+\by,M), \]
it's clear that the $E^2$ page is bounded and supported on $Y \cap V(\bx + \by)$, thereby giving rise to the relation
\[ \begin{array}{rcl} \ds \sum_{i=0}^{m+n}{(-1)^i[H_i(\bx+\by,M)]} & = & \ds \sum_{p,q \geq 0}{(-1)^{p+q}[E^2_{pq}]} \vspace{2mm}\\
                                                               & = & \ds \sum_{p=0}^{m}{\sum_{q=0}^{n}{(-1)^{p+q}[H_p(\bx,H_q(\by,M))]}} \end{array}\]
in $\GG_0(Y \cap V(\bx +\by))$. The right-hand side is, of course, $\Phi_{\bx}(\Phi_{\by}([M]))$.
\end{proof}

\begin{lem}\label{vanish}Let $\bx = (x_1, \cdots, x_m)$ and suppose that for some $i$ and $k$, $x_i^k\cdot M = 0$. Then $\Phi_{\bx}([M]) = 0$.
\end{lem}
\begin{proof}By Lemma \ref{composefilter}, $\Phi_{\bx}$ is just the composition of all of the $\Phi_{(x_j)}$, and the maps can be composed in any order. Thus, it suffices to show that $\Phi_{(x_i)}([M]) = 0$. The $x_i$-adic filtration on $M$ gives exact sequences
\[ 0 \to x_i^{p+1}M \to x_i^p M \to \frac{x_i^p M}{x_i^{p+1} M} \to 0. \]
If we denote by $N_p$, the rightmost term of this sequence, we see that the Koszul complex $K(x_i,N_p)$ has $0$-differential, whence $\Phi_{(x_i)}([N_p]) = 0$. Thus, $\Phi_{(x_i)}([x_i^p M]) = \Phi_{(x_i)}([x_i^{p+1} M])$ for all $p \geq 0$ and the conclusion follows.
\end{proof}

\begin{prop}Let $(A,\fm)$ be a Noetherian local ring and let $M$ be a $d$-dimensional $A$-module. Let $\bx = (x_1, \cdots, x_m)$ and put $Y = \Supp M$. Suppose that $\dim(Y \cap V(\bx)) = d-m$. Then there exists a $(d-m)$-dimensional module $N \in \mcm(Y \cap V(\bx))$ such that
\[ [N] = \Phi_{\bx}([M]) \hspace{5mm} \mbox{in $\GG_0(Y \cap V(\bx))$.} \]
\end{prop}
\begin{proof}We put $\by = (x_1, \cdots, x_{m-1})$. For the case of $m=1$ (i.e. $\by$ is an ``empty'' sequence), we shall declare $V(\by) = \Spec A$ and $\Phi_{\by}$ to be the identity on $\GG_0(Y)$. By induction, we may assume that there exists a $(d-m+1)$-dimensional module $N \in \mcm(Y \cap V(\by))$ such that
\[\Phi_{\by}([M]) = [N] \hspace{5mm} \mbox{in $\GG_0(Y \cap V(\by))$}. \]
From \ref{composefilter}, $\Phi_{(x_m)} \circ \Phi_{\by} = \Phi_{\bx}$, and hence,
\[ \Phi_{\bx}([M]) = \Phi_{(x_m)}([N]) \hspace{5mm} \mbox{in $\GG_0(Y \cap V(\bx))$}.\]

Consider the module $\Gamma_{(x_m)}(N) = \left\{u \in N : x_m^j u = 0 \mbox{ for some $j$} \right\}$. Since $N$ is finitely-generated, there exists a $k > 0$ such that $x_m^k \Gamma_{(x_m)}(N) = 0$. From the exact sequence
\[ 0 \to \Gamma_{(x_m)}(N) \to N \to N'' \to 0, \]
we know that since
\[ \dim(Y \cap V(\by)) = d-m+1 > \dim(Y \cap V(\by, x_m)), \]
$x_m$ lies outside of all primes $\fp \in Y \cap V(\by)$ such that $\dim(A/\fp) = d-m+1$. Thus, $\dim(\Gamma_{(x_m)}(N)) < d-m+1$, thereby forcing $\dim N'' = \dim N$. Using Lemma \ref{vanish} and noting that $x_m$ is, by definition, a non-zerodivisor on $N''$ gives the equation
\[ \Phi_{\bx}([M]) = \Phi_{(x_m)}([N]) = \Phi_{(x_m)}([N'']) = [N''/x_m N'']  \]
in $\GG_0(Y \cap V(\bx))$ where $N''/x_m N''$ is a genuine $(d-m)$-dimensional module in $\mcm(Y \cap V(\bx))$.
\end{proof}

\begin{rem}In the above proposition, we know that the class $[N] \in \GG_0(Y \cap V(\bx))$ is nontrivial: since $\dim N = \dim(Y \cap V(\bx))$, $[N]$ will evaluate to a positive number under the map $\GG_0(Y \cap V(\bx)) \to \mathbb{Z}$ given by
\[E \mapsto \sum_{\mathclap{\substack{\fp \in Y \cap V(\bx) \\ \dim(A/\fp) = d-m}}}{\ell(E \otimes A_\fp)}.\]
See Corollary \ref{serre2} below for a stronger statement.

Even if we omit the hypothesis that $\dim(Y \cap V(\bx)) = d-m$, we still have an equality
\[ [N] = \Phi_{\bx}([M]) \hspace{5mm} \mbox{in $\GG_0(Y \cap V(\bx))$.} \]
for some $N \in \mcm(Y \cap V(\bx))$. However, it can easily occur that $N = 0$ in this less restrictive case.
\end{rem}
\subsection{Hilbert-Samuel Multiplicity}
Let $(A,\fm)$ be a Noetherian local ring and let $M$ be a finitely-generated module. Suppose that $\fa \subseteq A$ is an ideal for which $\ell(M/\fa M) < \infty$. We define the Hilbert-Samuel multiplicity of $M$ with respect to $\fa$ via
\[ e_{\fa}(M) = \lim_{n \to \infty} \frac{d!}{n^d} \ell(M/\fa^n M) \hspace{5mm} (d = \dim M). \]
To obtain a function which is additive over exact sequences, we introduce, for each $r \leq \dim M$, the modified multiplicity function:
\[ e_{\fa}(M,r) =  \lim_{n \to \infty} \frac{r!}{n^r} \ell(M/\fa^n M) = \left\{ \begin{array}{ll} e_{\fa}(M) & r = \dim M \\ 0 & r > \dim M \end{array} \right. .\]
If $Y \subseteq \Spec A$ is a closed subset of dimension $r$ and $Y \cap V(\fa) = \left\{\fm \right\}$, then $e_{\fa}(-,r)$ is additive over short exact sequences in $\mcm(Y)$ \cite[II, Prop. 10]{Serre} and so defines a homomorphism $\GG_0(Y) \to \mathbb{Z}$. For $Z = \left\{\fm \right\}$, $\GG_0(Z)$ is simply the Grothendieck group on finite-length $A$-modules and $e_{\mathbf{0}}(-,0)$ coincides with the length function $\ell : \GG_0(Z) \to \mathbb{Z}$.

\begin{thm}\label{serre_mult}\cite[IV.3, Thm. 1]{Serre} Let $(A,\fm)$ be a Noetherian local ring and suppose that $M$ is a finitely generated module. Put $\bx = (x_1, x_2, \cdots, x_r)$ and suppose that $\ell(M/\bx M) < \infty$. Then
\[ e(M,r) = \sum_{p = 0}^{r}{(-1)^p \ell(H_p(\bx,M))}. \]
\end{thm}

\begin{cor}\label{serre2} Let $(A,\fm)$ be a Noetherian local ring and let $Y$ be a closed subset of $\Spec A$. Put $\bx = (x_1, x_2, \cdots, x_r), \bx' = (x_{r+1}, x_{r+2}, \cdots, x_{r+s})$ and suppose that $Y \cap V(\bx + \bx') = \left\{\fm \right\}$. Then there is a commutative diagram:
\[ \xymatrix{
\GG_0(Y) \ar[r]^{e_{\bx + \bx'}(-,r+s)} \ar[d]_{\Phi_{\bx}} & \mathbb{Z} \\
\GG_0(Y \cap V(\bx)) \ar[ur]^{e_{\bx'}(-,s)} \ar[r]^{\Phi_{\bx'}} & \GG_0(\left\{ \fm \right\}) \ar[u]^{\ell}
 } \]
\end{cor}
\begin{proof}Commutativity of the lower triangle follows at once from Theorem \ref{serre_mult}. Since $\Phi_{\bx'} \circ \Phi_{\bx} = \Phi_{\bx+\bx'}$ by Lemma \ref{composefilter}, the outer square commutes again by appealing to Theorem \ref{serre_mult}. The upper triangle is now forced to commute for formal reasons.
\end{proof}

\section{Proof of the Theorem}
\subsection{Dimension-One}
\begin{lem}\label{dim_one}Let $(B,\fm)$ be a complete Noetherian local ring of dimension $1$. Suppose that $B$ is reduced and has algebraically closed residue field. Then for any fixed prime $p \in \mathbb{Z}$, there is a principal ideal $I = (f)$ such that $e_I(B) = \ell(B/fB)$ is relatively prime to $p$.
\end{lem}
\begin{proof}By hypothesis, $B$ has no embedded primes, so if $f \in B$ is such that $\dim(B/fB) = 0$, then $f$ is necessarily a non-zerodivisor. Putting $I = (f)$, it's easily seen that $\ell(B/I^n B) = n \ell(B/fB)$, whence $e_I(B) = \ell(B/fB)$. We therefore want to find an $f$ for which $p$ does not divide $\ell(B/fB)$.

Denote by $S$ the set of non-zerodivisors of $B$, and let $Q = S^{-1}B$ be the total quotient ring of $B$. For any $f,g \in S$, the short-exact sequence
\[ 0 \to B/fB \stackrel{g}{\rightarrow} B/fgB \to B/gB \to 0 \]
permits us to define a monoid homomorphism $S \to \mathbb{Z}$ via $f \mapsto \ell(B/fB)$. This naturally extends to an ``order'' homomorphism $\ord_B: Q^{\times} \to \mathbb{Z}$ defined by $\ord_B(f/g) = \ell(B/fB) - \ell(B/gB)$ (cf. \cite[A.3]{Fulton}). It will suffice to show that $\ord_B$ is surjective, for if $\ell(B/fB) - \ell(B/gB) = 1$, then $p$ cannot divide both terms.

Let $\widetilde{B}$ be the normalization of $B$ inside of $Q$. Since $B$ is complete, the map $B \to \widetilde{B}$ is finite. $\widetilde{B}$ is therefore a semilocal, one-dimensional ring, and since $k = B/\fm$ is algebraically closed, the residue field of each closed point of $\widetilde{B}$ is isomorphic to $k$. It follows that for every finite-length $\widetilde{B}$ module $M$, it does not matter whether we measure its length over $B$ or $\widetilde{B}$: $\ell_{B}(M) = \ell_{\widetilde{B}}(M)$ (cf. \cite[A.1]{Fulton}). Since there is no ambiguity, we shall henceforth drop the subscripts.

Since $B$ and $\widetilde{B}$ share the same total quotient ring $Q$, we can define another order function $\ord_{\widetilde{B}}: Q^{\times} \to \mathbb{Z}$ with $\ord_{\widetilde{B}}(f/g) = \ell(\widetilde{B}/f \widetilde{B}) - \ell(\widetilde{B}/g \widetilde{B})$. We claim that $\ord_B = \ord_{\widetilde{B}}$. Since $S$ generates $Q^{\times}$ as an abelian group, it will suffice to show that both order functions agree on $S$. Given a non-zerodivisor $f \in B$, it will continue to be a non-zerodivisor in $Q$ and hence also in $\widetilde{B}$. This gives rise to the following diagram of finitely-generated $B$ modules:

\[ \xymatrix{
         & 0 \ar[d]          & 0 \ar[d]                      & K \ar[d] &                          \\
0 \ar[r] & B \ar[r] \ar[d]_f & \widetilde{B} \ar[r] \ar[d]_f & \widetilde{B}/B \ar[r] \ar[d]_f & 0 \\
0 \ar[r] & B \ar[r] \ar[d]   & \widetilde{B} \ar[r] \ar[d]   & \widetilde{B}/B \ar[r] \ar[d]   & 0 \\
         & B/fB              & \widetilde{B}/f \widetilde{B} & C                               & 
} \]

Since $B$ is reduced and hence has no embedded primes, the collection of all non-zerodivisors $S$ is just the complement of the union of all minimal primes of $B$. As $S^{-1}B = Q = S^{-1}\widetilde{B}$, we have that $S^{-1}(\widetilde{B}/B) = 0$, meaning that $\widetilde{B}/B$ vanishes at every minimal prime of $B$ and hence is supported in dimension $0$. As such,
\[ 0 \to K \to \widetilde{B}/B \stackrel{f}{\rightarrow} \widetilde{B}/B \to C \to 0 \]
is an exact sequence of finite-length modules, whence $\ell(K) = \ell(C)$. From the snake lemma, we have
\[ 0 \to K \to B/fB \to \widetilde{B}/f \widetilde{B} \to C \to 0 \]
which implies that $\ell(B/f B) = \ell(\widetilde{B}/f \widetilde{B})$ as desired.

Finally, we show that $\ord_{B} = \ord_{\widetilde{B}}$ is surjective. $\widetilde{B}$, by construction, is a one-dimensional, semilocal normal ring and hence the localization at each maximal prime is a DVR. In fact, since $\widetilde{B}$ is complete with respect to the linear topology defined by its Jacobson radical, it must decompose into a direct product of DVRs. By choosing an $h \in \widetilde{B}$ which generates one of the maximal ideals and lies outside of all of the others, we see that $\widetilde{B}/h \widetilde{B}$ is a field -- that is, $\ord(h) = 1$.
\end{proof}

\begin{exmp}\label{bad}Lemma \ref{dim_one} can fail if the residue field is not algebraically closed. Let $L = \mathbb{F}_2(t)$ where $t$ is a transcendental element. Put $A = L[[X,Y]]/(X^2+XY+Y^2)$. It is easily checked that $A$ is a domain. Since $A$ has infinite residue field, every $\fm$-primary ideal $I$ has a principal reduction $(f) \subseteq I$ (\cite[14.14]{Matsumura}). We claim that $\ell_A(A/fA) = e_{(f)}(A) = e_I(A)$ is divisible by $2$ for every $0 \neq f$.

If we denote by $\tilde{A}$ the normalization, the proof of Lemma \ref{dim_one} shows that $\ell_A(A/fA) = \ell_A(\tilde{A}/f \tilde{A})$. Since $X^2+XY+Y^2 = 0$ in $A$, dividing by $Y^2$ gives the relation $\left(\frac{X}{Y}\right)^2 + \left( \frac{X}{Y} \right) + 1 = 0$ in the fraction field. Thus, $\frac{X}{Y}$ lies in $\widetilde{A}$ and satisfies an irreducible polynomial over $L$. The residue field of $\widetilde{A}$ therefore contains a quadratic extension of $L$, and hence, $2$ divides $\ell_A(M)$ for all Artinian $\tilde{A}$-modules $M$ (cf. \cite[A.1.3]{Fulton}).

In particular, if $R = k[[T]]$ is any power-series subring of $A$ over which $A$ is finite, then we know \cite[VIII.10, Cor. 2]{ZSII} that $e_{(T)}(A)[L:k] = [K(A):K(R)]$. Thus, the degree of the fraction field extension $[K(A):K(R)]$ will always be divisible by $2$.
\end{exmp}

\begin{rem}\label{min_hyp}Strictly speaking, it is not absolutely necessary that $A$ have algebraically closed residue field for the conclusion of Lemma \ref{dim_one} to hold. One simply needs the residue fields at closed points of $A$ and $\widetilde{A}$ to all be isomorphic. \end{rem}

\subsection{General Case}
We begin by recalling the Cohen-Gabber theorem from the introduction but include some slight modifications to include the classically-known cases of equicharacteristic $0$ and mixed-characteristic.

\begin{thm}\cite[IV.2.1.1]{GTG}(Gabber) Let $A$ be a complete, equidimensional, reduced local ring. Suppose that the residue fields $k(\fp)$ for each minimal prime $\fp \subseteq A$ all have the same characteristic. Then there exists a subring $R \subseteq A$ such that
\begin{enumerate}
\item $R$ is a complete regular local ring.
\item The map $R \to A$ is finite and generically \'{e}tale.
\item $R$ and $A$ share the same residue field.
\end{enumerate}
\end{thm}

\begin{rem}As we mentioned in the introduction, if the residue fields at each minimal prime all have characteristic zero, the classical Cohen structure theorem guarantees a finite, injective map $R = V[[X_1, \cdots, X_m]] \to A$ where $V$ is a field or DVR of characteristic $0$ (cf. \cite[29.4]{Matsumura}). In this case being generically \'{e}tale is automatic since the fraction field of $R$ will have characteristic $0$ and hence be perfect. The nontrivial case where $A$ has equicharacteristic $p > 0$ is precisely what Gabber proved.
\end{rem}

\begin{thm}\label{gabber_mult}Let $(A,\fm)$ be a complete, reduced local ring of equidimension $d > 0$ with algebraically closed residue field. Suppose that the residue fields $k(\fp)$ for each minimal prime $\fp \subseteq A$ all have the same characteristic. Then for each prime $p \in \mathbb{Z}$, there exists a parameter ideal $I = (f_1, f_2, \cdots, f_d) \subseteq A$ such that $e_I(A)$ is relatively prime to $p$.
\end{thm}

\begin{proof}The idea is to reduce to the one-dimensional case where the result is known by Lemma \ref{dim_one}. By the Cohen-Gabber Theorem, we can find a regular local subring $R$ over which $A$ is finite and generically \'{e}tale. Choose $0 \neq h \in \fm$ such that $R_h \to A_h$ is \'{e}tale. Let $d = \dim A$. We now construct parameters $f_1, \cdots, f_{d-1}$ in $R$ such that
\begin{itemize}
\item[(a)] $R/(f_1, \cdots, f_{t})$ is a regular local ring for $0 \leq t \leq d-1$.
\item[(b)] $h,f_1, \cdots, f_{t}$ is a regular sequence in $R$ for $0 \leq t \leq d-1$.
\end{itemize}
When $t=0$, this just amounts to saying that $h$ is a non-zerodivisor on $R$, which is obvious. Assume that $f_1, \cdots, f_{t}$ have been constructed with $t < d-1$. Put $\overline{R} = R/(f_1, \cdots, f_t)$ and denote by $\fq$ its maximal ideal. By hypothesis, $h$ is a non-zerodivisor on $\overline{R}$, so that $\overline{R}/h\overline{R}$ is a Cohen-Macaulay ring of positive dimension, and hence, $\fq \notin \Ass_{\overline{R}}(\overline{R}/h\overline{R})$. By prime avoidance \cite[3.3]{Eisenbud} we have that $\fq \not\subset (\bigcup \operatorname{Ass}_{\overline{R}}(\overline{R}/h\overline{R})) \cup \fq^2$, meaning that we can choose some $\overline{f}_{t+1} \in \fq - \fq^2$ which is $\overline{R}/h\overline{R}$-regular. By lifting to an element $f_{t+1} \in R$, we we see that conditions (a) and (b) are satisfied.

Since $R \to A$ is finite, $A' := A/(f_1, \cdots, f_{d-1})$ is one-dimensional. Let $\fN$ be its nilradical and let $B = A'/\fN$. By Lemma \ref{dim_one}, there is a parameter $f_d \in A$ such that $\ell(B/f_d B)$ is not divisible by our fixed prime $p$. Note that since $B$ is reduced and has no embedded primes, $f_d$ is $B$-regular.

Put $\bof = (f_1, \cdots, f_{d-1})$ and $\bof'=(f_d)$. We claim that for the parameter ideal $I = \bof + \bof' = (f_1, \cdots, f_d)$, $e_I(A)$ is equal to $\ell(B/f_d B)$ and hence is relatively prime to $p$. Put $Y = \Spec A$, $Y' = Y \cap V(\bof) = \Spec A'$. From Corollary \ref{serre2} we have a diagram
\[ \xymatrix{
\GG(Y) \ar[r]^{e_I(-,d)} \ar[d]^{\Phi_{\bof}} & \mathbb{Z} \\
\GG(Y') \ar[r]^{\Phi_{\bof'}} & \GG(\left\{ \fm \right\}) \ar[u]^{\ell}
} \]
By the definition of $\Phi_{\bof}$ and the exact sequence $0 \to \fN \to A' \to B \to 0$, we have the relation
\[\Phi_{\bof}([A]) = [A'] + \sum_{i=1}^{d-1}{(-1)^i[H_i(\bof,A)]} = [B] + [\fN] + \sum_{i=1}^{d-1}{(-1)^i[H_i(\bof,A)]} \]
in $\GG_0(Y')$.

Since $R/(f_1, \cdots, f_{d-1}, h)$ is Artinian, so too is $A/(f_1, \cdots, f_{d-1},h) = A'/hA'$. Thus, $h$ lies outside of every minimal prime of the one-dimensional ring $A'$. In other words, if some finitely-generated $A'$ module $M$ is such that $M_h = 0$, then $M$ is supported only on the closed point and hence is killed by some power of $\fm$. Since $R_h \to A_h$ is \'{e}tale, then by base-change, so too is $\left(R/(f_1, \cdots, f_{d-1})\right)_h \to A'_h$. But $R/(f_1, \cdots, f_{d-1})$ is regular, meaning that $A'_h$ is regular (and, in particular, reduced). In other words, $\fN_h = 0$. Similarly, since $R_h \to A_h$ is flat, the $R$-regular sequence $f_1, \cdots, f_{d-1}$ must also be $A_h$-regular. In other words we obtain a vanishing of the Koszul homology: $H_i(\bof, A)_h = 0$ for $i > 0$. Thus, some power of $\fm$ --- and hence some power of $f_d$ --- kills each of the these modules, so by Lemma \ref{vanish},
\[ \Phi_{\bof'}([\fN]) = 0 \mbox{ and } \Phi_{\bof'}([H_i(\bof, A)]) = 0 \mbox{ for $i>0$}. \]
Since $f_d$ is a non-zerodivisor on $B$, we conclude that
\[ \begin{array}{rcl} \ds e_I(A,d) = \ell \circ \Phi_{\bof'} \circ \Phi_{\bof}([A]) & = & \ds \ell \circ \Phi_{\bof'} \left( [B] + [\fN] + \sum_{i=1}^{d-1}{(-1)^i[H_i(\bof,A)]} \right) \vspace{1mm}\\
																																										& =  & \ell(B/f_d B). \end{array}\]
\end{proof}

The proof of the sharpened Cohen-Gabber Theorem now follows immediately:

\begin{thm}\label{strong_cg}Let $(A,\fm,k)$ be a complete, equicharacteristic local domain of positive dimension whose residue field is algebraically closed. Then for any prime $p > 0$, there exists a regular subring $R = k[[X_1, \cdots X_d]]$ such that $R \to A$ is finite and $p$ is relatively prime to the generic degree $[K(A):K(R)]$.
\end{thm}
\begin{proof}Fix the prime $p$. By Theorem \ref{gabber_mult}, we know that $A$ admits a system of parameters $I=(f_1, \cdots, f_d)$ (with $d > 0$) such that $e_I(A)$ is relatively prime to $p$. The Cohen structure theorem \cite[28.3]{Matsumura} says that $k$ embeds into $A$ as a coefficient field (i.e. $k \to A \to A/\fm$ is an isomorphism). We define $R = k[[T_1, \cdots T_d]] \to A$ via $T_i \mapsto f_i$ and obtain a finite morphism which must be injective for dimensional reasons. By \cite[VIII.10, Cor. 2]{ZSII}, we know that $[K(A):K(R)] = e_I(A)$ and hence is relatively prime to $p$. 
\end{proof}

\begin{rem}Example \ref{bad} shows that this theorem can fail in dimension one if the residue field is not algebraically closed. In view of Remark \ref{min_hyp}, the proofs of Theorems \ref{gabber_mult} and \ref{strong_cg} will go through as long as we can be assured that after taking $A$ and going modulo the $f_1, \cdots f_{d-1}$, the residue field at the maximal ideal of $B = (A/(f_1, \cdots f_{d-1}))_{\operatorname{red}}$ is isomorphic to those of the normalization $\widetilde{B}$. Since the first $d-1$ parameters $f_1, \cdots, f_{d-1}$ are selected in a fairly generic fashion, one can ask whether there is a Bertini-type result that would allow one to choose the $f_i$ carefully enough so that the one-dimensional ring $B$ has the desired properties. If there were, we would be able to relax the requirement that $A$ have algebraically closed residue field as long as $\dim A \geq 2$.
\end{rem}

\subsection{The Mixed-Characteristic Case}
If $(A,\fm,k)$ is a complete, local domain of mixed-characteristic, one could ask for an analogue of Theorem \ref{strong_cg} where $R$ is a power-series over a DVR. In the presence of an additional regularity assumption, we have the following result:

\begin{thm}\label{strong_cg_mc}Let $(A,\fm,k)$ be a complete, local domain of mixed-characteristic having dimension $d > 0$ and whose residue field is algebraically closed. Let $q = \operatorname{char}(k)$ and assume $A/qA$ is generically reduced. Fix a prime number $p > 0$. Then
\begin{enumerate}
\item[(a)] There exists a parameter ideal $I = (q, f_1, \cdots, f_{d-1})$ such that $e_I(A)$ is relatively prime to $p$.
\item[(b)] There exists a regular subring $R = V[[X_1, \cdots, X_{d-1}]]$ such that
	\begin{enumerate}
	\item[(1)] $(V, q V, k)$ is a complete DVR with uniformizer $q$.
	\item[(2)] $R \to A$ is finite.
	\item[(3)] The generic degree $[K(A):K(R)]$ is relatively prime to $p$.
	\end{enumerate}
\end{enumerate}
\end{thm}
\begin{proof}We begin by remarking that if $d=1$, then the assumption that $A/qA$ is generically reduced implies that $A/qA$ is a field. Thus, $A$ is a DVR with uniformizer $q$, thereby making statements (a) and (b) trivialities. We shall henceforth assume that $d > 1$.

(a) Since $A$ is a domain, $\Phi_{(q)}([A]) = [A/q A]$ in $\GG_0(V(q))$. Let $\fN$ be the nilradical of $A/q A$ and put $C = (A/q A)/\fN$. Since $A/q A$ is generically-reduced, $\fN_{\fq} = 0$ for all minimal primes of $A/q A$, meaning that $\dim(\fN) < \dim(A/q A) = d-1$. As $C$ is reduced and equidimensional, Theorem \ref{gabber_mult} gives us $f_1, \cdots, f_{d-1} \in A$ which form a parameter system for $A/q A$ and $e_{(f_1, \cdots, f_{d-1})}(C)$ is relatively prime to $p$.

Put $\bof = (f_1, \cdots, f_{d-1})$ and let $I = (q) + \bof$. It is clear that $I$ is a parameter ideal for $A$; it will suffice to prove that $e_I(A) = e_{\bof}(C)$. From the exact sequence
\[ 0 \to \fN \to A/q A \to C \to 0 \]
we have that $e_{\bof}(A/q A,d-1) = e_{\bof}(C,d-1) = e_{\bof}(C)$ as $\dim(\fN) < d-1$. Now, by  Corollary \ref{serre2}, it follows that
\[
e_I(A) = \ell \circ \Phi_{\bof} \circ \Phi_{(q)}([A]) = \ell \circ \Phi_{\bof}([A/q A]) = e_{\bof}(A/qA, d-1) = e_{\bof}(C). \] 

For (b), we begin by noting that we are guaranteed such a DVR $V \subset A$ by the Cohen structure theorem (see, for example, \cite[29.3]{Matsumura}). By part (a), we can find a parameter ideal $I = (q, f_1, \cdots, f_{d-1}) \subset A$ such that $e_I(A)$ is relatively prime to $p$. The map $R = V[[X_1, \cdots, X_{d-1}]] \to A$ is defined by sending the $X_i$ to $f_i$. As in the proof of Theorem \ref{strong_cg}, this map is finite and injective. By \cite[11.2.6]{HunekeSwanson}, we have $e_I(A) = [K(A):K(R)]$.
\end{proof}

To illustrate why it is necessary to require that $A/qA$ be generically reduced, we present the following example:

\begin{exmp}Let $q \in \mathbb{Z}$ be a prime and let $(V, qV, k)$ be a DVR with algebraically closed residue field $k$. Let $n > 0$ and put $A = V[[X,Y]]/(qX-Y^n)$. We claim that for any parameter ideal of the form $I = (q,f)$, $e_I(A)$ will be divisible by $n$. Consequently, if $R \subset A$ is finite with $R$ a power-series over a complete DVR, then $[K(A):K(R)]$ will be divisible by $n$.

First note that $A/qA = k[[X,Y]]/(Y^n)$, so the $A/qA$-module isomorphism $(Y^i)/(Y^{i+1}) \cong k[[X,Y]]/(Y) \cong A/(q,Y)$ and the exact sequence
\[ 0 \to (Y^i)/(Y^{i+1}) \to k[[X,Y]]/(Y^{i+1}) \to k[[X,Y]]/(Y^i) \to 0 \]
together show that $[A/qA] = n[A/(q,Y)]$ in $\GG_0(V(q))$. By Corollary \ref{serre2},
\[e_I(A) = \ell \circ \Phi_{(f)} \circ \Phi_{(q)}([A]) = \ell \circ \Phi_{(f)}([A/qA]) = n \cdot \ell \circ \Phi_{(f)}([A/(q,Y)]). \]
\end{exmp}

\section*{Acknowledgments}
I originally proposed this as one of several open problems at a 2016 Summer REU that I jointly organized with K. Tucker and W. Zhang at UIC. I would like to thank the students for their spirit and enthusiasm. In particular, Example \ref{bad} is based on an observation of J. Mundinger which showed that any automorphism of $\mathbb{F}_2[[X,Y]]$ will take $X^2 + XY + Y^2$ to itself (mod $\fm^3$). I would also like to thank the anonymous referee for their thoughtful comments and for catching a few mistakes in the original draft. The final section, concerning generalizations to the case of mixed-characteristic, was also added in response to a question raised by the referee.
This work was partially supported by an NSF RTG Grant (DMS-1246844).

\bibliographystyle{amsalpha}
\renewcommand{\bibfont}{\normalsize}
\renewcommand{\refname}{}
\bibliography{koszul}

\end{document}